\documentclass[reqno,12pt]{amsart}
\usepackage{mathtools}
\usepackage{amsmath}
\usepackage{amssymb,latexsym}
\usepackage{fullpage}
\usepackage{color}
\theoremstyle{plain}

\newtheorem{corollary}{Corollary}[section]
\newtheorem{lemma}{Lemma}[section]
\newtheorem{proposition}{Proposition}[section]
\theoremstyle{remark}

\newtheorem{remark}{Remark}[section]

\numberwithin{equation}{section}

\newcommand{\bC}{{\mathbb C}}

\newcommand{\bZ}{{\mathbb Z}}
\newcommand{\bN}{{\mathbb N}}

\newcommand{\B}{\mathcal{B} }
\newcommand{\Fe}{\mathcal{F} }

\renewcommand{\>}{\rangle}

\begin{document}

\title{Generating functions for symmetric and shifted symmetric functions}

\author{Naihuan Jing}
\address{Department of Mathematics, North Carolina State University, Raleigh, NC 27695, USA}
\email{jing@math.ncsu.edu}
\author{Natasha Rozhkovskaya}
\address{Department of Mathematics, Kansas State University, Manhattan, KS 66502, USA}
\email{rozhkovs@math.ksu.edu}
%\newpage
%\msc{,,,,}    %<-------------------
\thanks{}%Supported in part by Simons Foundation and NSFC}

\begin{abstract}
We describe  generating  functions  for  several important families of classical  symmetric functions and shifted  Schur  functions.
The approach is originated from vertex operator  realization  of symmetric functions and offers a unified method to
treat various families of symmetric functions and their shifted analogues.
\end{abstract}

\maketitle

%%%%%%%%%%%%%%%%%%%%%%%%%%%%

\section{Introduction}
%\subsection{Classical boson-fermion  correspondence}\label{bfc}
This paper pursues   two  goals.
The first one  is  to describe several important  families of  symmetric functions and their analogues  through  generating  functions. With such an  approach,  one could  argue that all  information  about a particular family  of  symmetric  functions   is encoded in  the  defining factor   of the  generating  function, which we call  ``correlation function'' below.   For example,  all the  properties of  Schur symmetric functions $s_{\lambda}$ can be  reconstructed from  the  correlation  function  $f(x)= 1-x$, and the  generating  function  in this case  has the form
 \begin{align*}
\det\left[ u_i^{ i-j}\right]\prod_{i=1}^{l} Q(u_i)=\sum_{l(\lambda)\le l} s_\lambda u_1^{\lambda_1}\cdots u_l^{\lambda_l},
  \end{align*}
where $Q(u)$ is the generating function of the homogeneous complete symmetric functions (for detail see the next sections).  %(\ref{fsym});  are given  by  (\ref{fq})  and (\ref{fHL})  respectively.)
 This approach, extended to  the shifted Schur  functions $s^*_{\lambda}$ \'a la
 Okounkov-Olshanski \cite{OO1}, allows  us to  find a {\it new formula} for their   generating  function:
\begin{align*}
\det\left[ \frac{1}{(u_i| i-j)}\right]\prod_{i=1}^{l} Q^*(u_i-i+1)
 =\sum_{l(\lambda)\le l}\frac{s^*_\lambda}{(u_1|\lambda_1)\cdots (u_l|\lambda_l)},
\end{align*}
where $Q^*(u)$ is the generating function of the homogeneous complete shifted symmetric functions.

 The  origin  of  this   point of view    lies  in  the renown  vertex operator  realization of  Schur functions that  arises  naturally  in the framework of  the boson-fermion correspondence. This motivates  our second  goal to describe in a  unified way the interpretation of  the generating  functions   as a result of a successive  application of  vertex operators to  a vacuum vector.

The results of this paper are related to  our earlier  work \cite{JR}, where  we  observed that  the statement of the  boson-fermion correspondence can be  deduced   directly  from  the  Jacobi\,--\,Trudi  identity (\ref{e:JT2}) or, equivalently, Giambelli formula,  -- the  classical statements of  the  theory of symmetric functions.  This observation motivated us in  \cite{JR}   to apply  a generalized version  of Jacobi\,--\,Trudi identity for the  construction of  different  versions of the   action of the special infinite dimensional Clifford algebra on the polynomial algebras that  arise in representation theory  as analogues of the algebra of symmetric functions.  In particular, we obtained  explicit  formulas for the vertex operators   corresponding to  universal characters of classical Lie algebras, shifted Schur functions and  generalized  Schur symmetric functions associated to linear recurrence relations.

 In this  note  we take  a  related  but   different direction. We  observe that  several  important  families of  symmetric  functions can be interpreted as  coefficients of the generating  functions  defined by  certain  correlation factors.
 Then there is a simple  and a natural way to  introduce the action of the Clifford algebra on the  vector space  spanned by  the coefficients of the  generating  function, which in turn  gives rise to (in some cases modified) vertex operators.
   In particular, this  gives  vertex operator realization  for  several  generalizations of  symmetric  functions not  covered by  determinant-like  formula of the Jacobi\,--\,Trudi type (e.g.  Schur $Q$-functions  and Hall-Littlewood  symmetric  functions).

The paper is organized as follows.  In Section \ref{Class}
% \ref{Def}  -- \ref{Qsch}
we introduce   the  general construction and   illustrate  it   with the examples of   Schur functions and  Schur Q-functions. This set up immediately  allows  us  to define  creation and  annihilation  operators,     which  is  done in  Section \ref{Oper}  along with ``normally  ordered'' decomposition  of  the corresponding  vertex operators. This is further illustrated by the
construction of the Hall-Littlewood symmetric functions.
In Section \ref{Shift} we apply the same approach  to  find the generating  function for  shifted  Schur functions and its vertex operators presentation. In the last   Section \ref{ap} for the  convenience of  reader we  outline the main components  of  boson-fermion  correspondence and review some  properties of  symmetric  functions used in the note.

\section{Symmetric functions}\label{Class}
\subsection{Definitions}\label{Def}

Let $\B=K[Q_1, Q_2,\dots] $  be the ring of  polynomials in infinitely many commuting variables, where  $K=\bC$ or $\bC[t]$ in the  examples  that  follow.   For convenience, set $Q_0=1$ and $Q_i=0$ for $i<0$.
Write the generating  function
%\begin{align*}
     $	 Q(u)= \sum_{k=0}^{\infty} {Q_k}{u^k}, $
		%\end{align*}
 and set
%\begin{align*}
  	$  R(u)= \sum_{k=0}^{\infty} {R_k}{u^k}$
	  %\end{align*}
to be  the formal  inverse series of $Q(u)$:
\begin{align}\label{QR}
Q(u)R(u)= R(u)Q(u)=1.
  \end{align}
Note that the coefficients  $R_k$'s are polynomials in $Q_l$'s, and   $R_0=1$. Set also $R_i=0$  for $i<0$.
The  example of  Schur functions  considered  below in   Section \ref{Sch1}  will  explain  the expected  meaning of  elements $Q_i$'s and $R_i$'s  in the context of   symmetric functions.

 Let  $f\in \bC[[x]]$
be  an invertible  power series in  $x$:
\begin{align}\label{f_set}
 f(x)= \sum_{k=0}^{\infty}f_k x^k ,\quad
 f^{-1}(x)= \sum_{k=0}^{\infty}\tilde f_k x^k.
\end{align}
   Define
\begin{align}
Q(u_1,u_2,\dots, u_l)&=\prod_{i<j} f(u_j/u_i) \prod_{i=1}^{l} Q(u_i),\label{e:Qdef}\\
R(u_1, u_2,\dots, u_l)&=\prod_{i<j} f(u_j/u_i)\prod_{i=1}^{l} R(u_i),\label{e:Rdef}
\end{align}
and  for an ordered $k$-tuple of  integers  $\lambda= (\lambda_1,\dots,\lambda_k)$  set
\begin{align}
Q_\lambda: = Q(u_1,u_2,\dots, u_l)|_{{u_1^{\lambda_1}\dots u_l^{\lambda_l}}},\notag\\
R_\lambda: = R(u_1,u_2,\dots, u_l)|_{{u_1^{\lambda_1}\dots u_l^{\lambda_l}}}.\notag
\end{align}

We think of $Q(u_1,u_2,\dots, u_l)$ and $R(u_1,u_2,\dots, u_l)$  as the generating functions  for   coefficients $Q_\lambda$ and $R_\lambda$ respectively.

\begin{lemma}%%%%%%%%%%%%%%%%%%%%%
The elements $Q_\lambda$ and $R_\lambda$ are well-defined by  (\ref{e:Qdef}, \ref{e:Rdef})  as polynomials in  variables $Q_1,Q_2,\dots$.
\end{lemma}
\begin{proof}
This is proved by induction.   The statement is  true for  coefficients  of  $Q(u)$  just  by the definition.
We have:
\begin{align}
Q(u_1, \dots, u_l,v)&= \prod_{i=1}^{l}f(v/u_i) Q(u_1,\dots,  u_l) Q(v)\notag\\
&= \sum_{\lambda_i\in\bZ}\sum_{k=0}^{\infty}
 \sum_{ r_1,\dots, r_l=0} ^{\infty}
 f_{r_1}\cdots f_{r_l} Q_{(\lambda_1,\ldots,\lambda_l)}Q_k
u_1^{\lambda_1-r_1}\cdots u_l^{\lambda_l-r_l}
v^{k+r_1+\dots + r_l}\notag\\
 &=\sum_{a_i\in \bZ} \sum_{m=0}^{\infty} \sum_{ r_1,\dots, r_l=0} ^{\infty} f_{r_1}\cdots f_{r_l}
 Q_{(a_1+r_1,\ldots, a_l+ r_l)}Q_{m- r_1-\cdots - r_l} u_1^{a_1}\cdots u_l^{a_l} v^{m}.\label{e01}%\label{e01}
\end{align}
Since $Q_k=0$ for $k<0$, we get $0\le r_1+\dots+ r_l\le m$, and the coefficient of $u_1^{a_1}\dots u_l^{a_l} v^{m}$  in (\ref{e01})
is a sum of polynomials in $Q_n$'s with finitely many nonzero terms.
%--  coefficient is  finite sum  over ${r_1,\dots r_l\ge 0,\quad  r_1+\dots r_l\le m}$.
For $R( u_1,\dots, u_l)$  the argument is exactly the same.

\end{proof}
\begin {remark}
We restricted ourselves to the case when the generating  function for $Q_\lambda$'s  is  determined by  an invertible series  $f(v/u)$  in non-negative powers of  $(v/u)$.
One can ask, if   this definition  can be  extended  to a more  general  function  $f(u,v)$ by setting
\begin{align*}
Q(u_1,u_2,\dots, u_l):=\prod_{i<j} f(u_i,u_j) \prod_i Q(u_i).
\end{align*}
Note that in such a case certain restrictions on $f$ would  follow right away.
Indeed,  assume formally  that $f(u,v)= \sum_{k\in \bZ} \sum_{l\in \bZ} f_{k,-l} u^{k}v^{-l}$. Write
\begin{align*}
Q(u,v)&= f(u,v) Q(u) Q(v)=\sum_{k,l\in \bZ} \sum_{s, t=0}^{\infty} f_{k,-l} Q_sQ_t u^{k+s}v^{-l+t}\\
&=\sum_{m,n\in \bZ}\left( \sum_{s, t=0}^{\infty} f_{m-s,n-t} Q_{s}Q_{t}\right) u^{m}v^{n}.
\end{align*}
Then in general,  one has to impose some conditions on $f(u,v)$ to make
 the sum in  the parentheses   finite.  The existence of  well-defined  series $R(u_1,\dots, u_l)$  and the  construction of well-defined operators  $DQ_k$ and $DR_k$  as in  Section \ref{adj}  below would  imply further  restrictions.
\end{remark}

We will often   use  abbreviated  notation $\bar u$  for  multivariable argument $(u_1,\dots, u_l)$ of functions:
$$Q(\bar u)=Q(u_1,\dots, u_l),\quad  \quad R(\bar u)=R(u_1,\dots, u_l),$$
$$Q(v,\bar u)=Q(v,u_1,\dots, u_l),\quad \text{ and} \quad R(v,\bar u)=R(v,u_1,\dots, u_l).$$

%%%%%%%%%%%%%%%%%%%%%%%%%%%%%%%%%%%%%%%%%%%%%%%%%
\subsection{Generating function for Schur symmetric  functions.} \label{Sch1}
The celebrated  Schur functions (see Appendix \ref{A-sch})  serve as  the main example of functions  with   a generating  series of the form (\ref{e:Qdef}). Namely, set
\begin{align}
f(x)= 1-x.\label{e:f1}%\label{e:f1}
\end{align}
Then
\begin{align*}
Q(\bar u)&=\prod_{i<j}  \left(1-\frac{u_j}{u_i}\right) \prod_{i=1}^{l} Q(u_i)
= \det[u_i^{i-j}] \prod_{i=1}^{l} Q(u_i)
= \det [u_i^{i-j}Q(u_i) ]\\
&=\sum_{\sigma\in S_l} sgn(\sigma) \sum_{a_1\dots a_l} Q_{a_1}u_1^{a_1+1-\sigma(1)}\cdots Q_{a_l}u_l^{a_l+l-\sigma(l)}\notag\\
&=
\sum_{\lambda_1\dots \lambda_l} \sum_{\sigma\in S_l}sgn(\sigma)Q_{\lambda_1-1+\sigma(1)}\cdots Q_{\lambda_l-l+\sigma(l)}u_1^{\lambda_1}\cdots u_l^{\lambda_l}\notag
\\
&=\sum_{\lambda_1\dots \lambda_l}\det [Q_{\lambda_i-i+j}]u_1^{\lambda_1}\cdots u_l^{\lambda_l}.\notag
\end{align*}
Then  the Jacobi\,--\,Trudi  identity for  Schur symmetric  functions  (\ref{e:JT1})
implies that
  the identification  of generators  $Q_k$ with the complete  symmetric functions $h_k$ defined  by (\ref{defh}) leads to that of $Q_\lambda$ with the
  Schur symmetric  functions $s_\lambda$ for  any {\it partition} $\lambda$ with at most $l$ parts.

  Moreover,  observe that one can use  the Jacobi\,--\,Trudi  identity to extend the definition of $s_\lambda$
 to  any {\it integer vector}.   Namely, define
 \begin{align}\label{salph}
 s_\alpha =\det[h_{\alpha_i-i+j}]_{1\le i,j\le l}
 \end{align}
  for any   vector  $\alpha = (\alpha_1, \alpha_2,\dots ,\alpha_l)$    with   entries $\alpha_i \in \bZ$ .
 Then it is  clear that
\begin{align}
 s_{(\dots, \alpha_i, \alpha_{i+1},\dots)}= - s_{(\dots, \,\alpha_{i+1} -1\,, \,\alpha_{i}+1\, ,\dots)}.\label{lowr1}
\end{align}
This follows  form permutation of the rows of the determinant (\ref{salph}).
Let $\rho_l= (l-1, l-2,\dots, 0)$, where $l$  is such a number that  at most $l$  entries of the integer vector  $\alpha$  are non-zero.
It is easy to see that
$s_\alpha\ne 0$  if and only if $\alpha -\rho_l = \sigma (\lambda -\rho_l)$ for some  permutation $\sigma \in S_l$  and some  partition $\lambda =(\lambda_1\ge \lambda_2\ge \dots\ge \lambda _l\ge 0)$,   ($\lambda_i\in \bZ_{\ge0}$).   In such a case,
\begin{align}
s_\alpha = sgn(\sigma) s_{\lambda}.\label{lowr2}
\end{align}
For example, $s_{(1,3)}=s_{(1,3,0)}=- s_{(2,2)} = s_{(2,-1,3)}= - s_{(1,-1,4)}$.

\bigskip

  {\it Thus,  through  identification $Q_k= h_k$, where $h_k$  are complete  symmetric  functions (\ref{defh}), the series
  \begin{align}
Q(\bar u)&=\prod_{i<j}  \left(1-\frac{u_j}{u_i}\right) \prod_{i=1}^{l} Q(u_i)=\sum_{l(\lambda)\le l} s_\lambda u_1^{\lambda_1}\cdots u_l^{\lambda_l}\label{fsym}
  \end{align}
 can be interpreted as  the generating function for Schur symmetric  functions indexed by   integer vectors $\lambda$  with  $l(\lambda)$  non-zero parts, where $l(\lambda)\le l$ }

Similarly,  identifying $R_k=(-1)^k e_k$,  where  $e_k$  is the $k$-th elementary symmetric function (\ref{defe}),   we deduce from    (\ref{e:JT2}) that
   nonzero  coefficients   of the series $R(\bar u)$  also  coincide up to a sign  with the Schur symmetric functions: if  $\lambda$  is a  partition of $N$,  then $R_\lambda= (-1)^N s_{\lambda^\prime}$,  where $\lambda^\prime$  is the conjugate  to $\lambda$  partition.

%%%%%%%%%%%%%%%%%%%%%%%%%%%%%%%%%%%%%%%%%%%%%%%%%

\subsection{Generating function for Schur  Q-functions.} \label {Qsch}
Set
\begin{align*}
f(x)=\frac{1-x}{1+x}.
\end{align*}
Then
\begin{align}\label{QQsch}
Q( u_1,\dots, u_{2l})=\prod_{1\le i<j\le 2l} \frac{u_i-u_j}{u_i+u_j} \prod_{i=1}^{2l} Q(u_i).
\end{align}
Recall that the Pfaffian  of a skew-symmetric  matrix $A=(a_{ij})$ of size $2l\times 2l$ is defined as
$$
\mathrm{Pf}(A)=\sum_{\sigma\in S'_{2l}}  sgn(\sigma) a_{\sigma(1) \sigma(2)}  \cdots a_{\sigma(2l-1) \sigma(2l)},
$$
where $S'_{2l}$ consists of $\sigma\in S_{2l}$  such that $\sigma(2k-1)<\sigma(2k)$ for $1\le k\le l$ and
$\sigma(2k-1)<\sigma(2k+1)$ for $1\le k\le l-1$. % p.254
For $|x|<1$ the  function  $f(x)$  has an expansion
$$
f(x)=\frac{1-x}{1+x}= 1+2\sum_{k=1}^{\infty} (-1)^kx^k,
$$
hence  for $|v|<|u|$,
\begin{align*}
Q(u,v)=\frac{u-v}{u+v}  Q(u)Q(v)=1+2\sum_{k=1}^{\infty} (-1)^k\frac{v^k}{u^k} Q(u)Q(v)=
\sum_{\lambda_1,\lambda_2} {Q_{\lambda_1,\lambda_2}}{u^{\lambda_1} v^{\lambda_2}},
\end{align*}
 where
 \begin{align*}
 Q_{\lambda_1, \lambda_2}=Q_{\lambda_1}Q_{\lambda_2}+ 2\sum_{s=1}^{\infty} (-1)^{s}Q_{\lambda_1+s}Q_{\lambda_2-s}.
\end{align*}
Note that the later sum has finitely many non-zero terms.

Denote $ f_{ij}=\frac{u_i-u_j}{u_i+u_j}$
and  observe that
$$
\mathrm{Pf} \left [\frac{u_i-u_j}{u_i+u_j} \right] _{i,j=1,\dots, 2l}=\prod_{1\le i<j\le 2l}  \frac{u_i-u_j}{u_i+u_j}.
$$
 Then
we can write
\begin{align*}
\mathrm{Pf} [Q (u_i,u_j)]_{i,j=1,\dots, 2l}&=\mathrm{Pf} [f_{ij} Q (u_i) Q(u_j)]_{i,j=1,\dots, 2l}\\
&=  \sum_{\sigma\in S'_{2l}}sgn(\sigma)  f_{\sigma(1)\sigma(2)} Q(u_{\sigma(1) }) Q(u_{\sigma(2) }) \cdots f_{\sigma(2l-1)\sigma(2l)} Q(u_{\sigma(2l-1) }) Q(u_{\sigma(2l) })\\
&=  \sum_{\sigma\in S'_{2l}} sgn(\sigma) f_{\sigma(1)\sigma(2)} \cdots f_{\sigma(2l-1)\sigma(2l)} \prod_{i=1}^{2l} Q(u_i)\\
%=Pf [f_{ij}] _{i,j=1,\dots 2l} \prod_{i=1}^{2l} Q(u_i)\\
&= \mathrm{Pf} \left [\frac{u_i-u_j}{u_i+u_j} \right] \prod_{i=1}^{2l} Q(u_i)=\prod_{i<j}  \frac{u_i-u_j}{u_i+u_j} \prod_{i=1}^{2l} Q(u_i)\\
&=Q(u_1,u_2,\dots, u_{2l}).
\end{align*}
Therefore we have obtained that
\begin{align*}
Q(u_1,u_2,\dots, u_{2l})=\mathrm{Pf} [Q (u_i,u_j)]_{i,j=1,\dots, 2l}
\end{align*}
But  we also  see that
\begin{align*}
\mathrm{Pf} [Q (u_i,u_j)]_{i,j=1,\dots, 2l}&=
\mathrm{Pf}[ \sum_{\lambda_i,\lambda_j} {Q_{\lambda_i,\lambda_j}}  u_i^{\lambda_i}  u_j^{\lambda_j} ]_{i,j=1,\dots, 2l}\\
&=  \sum_{\sigma\in S'_{2l}}sgn(\sigma)  \sum_{\lambda_i,\lambda_j}Q_{\lambda_{\sigma(1)},\lambda_{\sigma(2)}}  u_{\sigma(1)}^{\lambda_{\sigma(1)} }
 u_{\sigma(2)}^{\lambda_{\sigma(2) }}\cdots Q_{\lambda_{\sigma(2l-1)},\lambda_{\sigma(2l)}} u_{\sigma(2l-1)}^{\lambda_{\sigma(2l-1)} }
 u_{\sigma(2l)}^{\lambda_{\sigma(2l)} }
\\
&=  \sum_{\sigma\in S'_{2l}} sgn(\sigma) \sum_{\lambda_i,\lambda_j}Q_{\lambda_{\sigma(1)},\lambda_{\sigma(2)}}  \cdots  Q_{\lambda_{\sigma(2l-1)},\lambda_{\sigma(2l)}} u_{1}^{\lambda_{1} }\dots
 u_{2l}^{\lambda_{2l} }
\\
&=\sum_\lambda \mathrm{Pf} [Q_{\lambda_i,\lambda_j}]u_1^{\lambda_1}\cdots u_{2l}^{\lambda_{2l}}.
\end{align*}
Therefore,  the  coefficient   of $u_1^{\lambda_1}\cdots u_{2l}^{\lambda_{2l}}$ in  (\ref{QQsch}) is of the form
$$
Q_\lambda = \mathrm{Pf} [Q_{\lambda_i, \lambda_j}]_{i,j=1,\dots 2l},
$$
which  for  strict partitions $(\lambda_1>\dots>\lambda_{2l}\ge 0)$  gives  one  of the definitions   of Schur $Q$-functions (see  Section \ref{App:q}  for the summary  of the  properties).
{\it Thus, identification $Q_k$  with Schur $Q$-polynomials  leads to the interpretation of
\begin{align}
\mathrm{Pf} [Q (u_i,u_j)]=\prod_{1\le i<j\le 2l} \frac{u_i-u_j}{u_i+u_j} \prod_{i=1}^{2l} Q(u_i) \label{fq}
\end{align}
  as  a  generating function for Schur $Q$-functions $Q_\lambda$'s. }
%%%%%%%%%%%%%%

\begin{remark}
 Note that   from (\ref{def:q}), $Q(u)Q(-u)=1$. Hence for Schur $Q$-functions, $R(u)=  Q(-u)$.
 \end{remark}
\begin{remark}
 Generating  functions  (\ref{fsym}) and   (\ref {fq}) along with  the generating  function  for Hall-Littlewood   polynomials  (\ref{fHL})  discussed below can be found in  \cite{Md}  III - (8.8) and  III -  (2.15).
\end{remark}

%%%%%%%%%%%%%%
\section{Operators.}   \label{Oper}
\subsection{Creation and annihilation  operators}\label{caop}
We refer to  Section \ref{BF}  for the outline of classical boson-fermion  correspondence which provides the motivation  for our constructions. The correspondence  involves the  action of Clifford algebra  (fermions) on  Fock  boson  space through  vertex operators $\Psi^\pm (u)$.
Note that   traditional formulas for vertex operators have  a ``normally ordered  form'':  they  are decomposed into a
product of two generating functions, which separate the action of differentiations and
of multiplication operators. This decomposition is important for further applications of
vertex operators.

 Our next goal  is to  introduce  families of  operators acting on the coefficients of the  generating  functions $Q(\bar u)$ (or  $R(\bar u)$) that generalize the classical  action of Clifford algebra  on  Fock  boson  space. It turns out, that in  this setting  these operators have very  simple  meaning.
 In case of  the  generating  function for Schur symmetric functions  we  get  back   vertex operators of    classical boson-fermion  correspondence.

In the setting  of  Section \ref{Def}, consider the  operators  $\Psi^{+}_k$ $(k\in \bZ)$    acting on the  set  of  coefficients $Q_\lambda$'s and $\Psi^{-}_k$  $(k\in \bZ)$, acting on the  set of coefficients
 $R_\lambda$'s by the rule
\begin{align}
\Psi^+_k (Q_\lambda)= Q_{(k,\lambda)}, \quad
\Psi^-_{-k} (R_\lambda)= R_{(k,\lambda)}.\label{pq1}
\end{align}
Set
\begin{align*}
\Psi^+(v)=\sum_{k\in \bZ}\Psi^+_k v^k, \quad
\Psi^-(v)=\sum_{k\in \bZ}\Psi^-_{-k} v^{k}.
\end{align*}
Equivalently, the action of $\Psi^\pm_k$  is defined in terms of  generating functions:
\begin{align}
\Psi^+(v)Q(u_1,\dots,  u_l)= Q(v,u_1,\dots, u_l),\label{psi+}\\
\Psi^-(v)R(u_1,\dots,  u_l)= R(v,u_1,\dots, u_l).\label{psi-}
\end{align}
Note that  {\it   generating functions $Q(\bar u)$  and $R(\bar u)$ can be interpreted as  a result of  successive applications  of operators  $\Psi^\pm(u_i)$  to a vacuum vector:}
\begin{align*}
Q(u_1,\dots u_l)= \Psi^+(u_1)\circ \dots\circ \Psi^+(u_l) \, ( 1),\\
R(u_1,\dots u_l)= \Psi^-(u_1)\circ \dots\circ \Psi^-(u_l) \, (1).
\end{align*}
In cases when  the coefficients  $Q_\lambda$'s
 (and/or $R_\lambda$'s)   span the whole space $\B$ or form a linearly independent spanning set when $\lambda$ runs through a subset of partitions,   formulas  (\ref{psi+})  and (\ref{psi-})
 define the   action of  operators  $\Psi^\pm_k$ on  the whole space $\B$. This condition   holds for   Schur functions,  Schur  Q-functions and  Hall-Littlewood symmetric functions considered  below.
%%%%%%%%%%%%%%%%%%%%
\subsection{Adjoint operators} \label{adj}
We   want to  write  ``normally  ordered  decomposition''  of $\Psi^\pm(u)$ similar to decomposition  (\ref{psi1}), (\ref{psi11}).
We introduce   operators
  $DR_m$   acting on $\B$ as follows.
  Let  $f(x)$  be the function  of the form  (\ref{f_set})  that  we  normalize by setting  $f_0=1$.
Collect  operators $DR_m$  in a  generating function
   \begin{align}
DR(u)=\sum_{m=0}^{\infty}\frac{DR_m}{u^m},
   \end{align}
and define  the action of  operators $DR_m$  on the  generators  $Q_k$'s by the formula
 \begin{align}
    DR(u)\, \big(Q(v)\big)= f(v/u) Q(v).\label{e:a0}
     \end{align}
The relation   (\ref{e:a0}) is equivalent to the condition
 \begin{align*}
 &DR_m(Q_k)=f_m Q_{k-m} \quad (m=0,1,2,\dots ,\quad  k=1,2,\dots ), \\
  &DR_m(1)= 1.
  \end{align*}
We  extend the action of $DR_m$ to  $\B$  by the following  rule.
For any $P, S\in \B$, set
   \begin{align}
 \quad DR_m(P S)= \sum_{k+l=m} DR_k(P) DR_l(S),\label{e:a2}\\
 \quad DR_m(P +S)=  DR_m(P) +DR_m(S).\label{e:a3}
 \end{align}
In particular, (\ref{e:a2}) and (\ref{e:a3}) imply
\begin{align}
DR(u)\, (PS)\, = DR(u)(P)\, DR(u)(S).\label{a12}
\end{align}

Set  also
$$
DQ(u)=\sum_{k=0}^{\infty}\frac{DQ_k}{u^k}
$$
 to be the inverse formal series of $DR(u)$:
\begin{align}\label{e:a11}
 DQ(u)\circ DR(u)=Id.
\end{align}
Note that
operators $DR_k$ commute with each other, $DR_0=Id$, and the coefficients  of $DQ(u)$
are polynomials in $DR_k$. Thus,  $DQ_k$ also  commute  with each other and with   $DR_k$, and, in particular, we can also  write   $DR(u)\circ DQ(u)= Id$.

We summarize the  basic  properties of  $DR(u)$ and $DQ(u)$  in the following  lemma:
%%%%%%%%%%%%%%%%%%%%
\begin{lemma}
Along with  defining properties (\ref{e:a0})  and (\ref{a12}) one has
\begin{align}
    %  DR(u)\,(Q(v))= f(v/u) Q(v),
      \quad DQ(u)\,(Q(v))&= f^{-1}(v/u) Q(v),  \label{e:a22}\\%
     DR(u)\,(R(v))= f^{-1}(v/u) R(v),
  &\quad
DQ(u)\,(R(v))= f(v/u) R(v), \label{e:a21}\\
  %  DR(u)\, (PS)\,& = DR(u)(P)\, DR(u)(S),\quad %\label{a7}\\
DQ(u)\, \big(PS\big)\,& = DQ(u)\big(P\big)\, DQ(u)\big(S\big).\label{a8}
\end{align}
\end{lemma}
\begin{proof}
 Formula in (\ref{e:a22})  is  implied by
\begin{align*}
Q(v) =DQ(u)\circ DR(u) \,\big(Q(v)\big)= DQ(u) (f(v/u) Q(v)).
\end{align*}
Equalities of (\ref{e:a21})  follow from (\ref{e:a22}), (\ref{a12}) and (\ref{QR}),  and  for a proof of (\ref{a8})
 write
\begin{align*}
&DQ(u)\big(P\big)\cdot\, DQ(u)\big(S\big)=Id\, \left( DQ(u)(P)\, \cdot DQ(u)(S)\right)\\
&=(DQ(u)\,\circ DR(u))\, \left( \, DQ(u)(P)\, \cdot DQ(u)(S)\,\right) \\
&=DQ(u)  \big(DR(u)  \circ DQ(u) (P)\,\cdot\,DR(u)\circ DQ(u)(S)\big)=\, DQ(u) (P\cdot S).
\end{align*}
\end{proof}

\begin{remark}
The explicit action of    operators $DR_m$ and $DQ_l$  can be obtained by expanding the formulas (\ref{e:a22}),  (\ref{e:a21}) in  powers $u$ and $v$.
For example,  in case of  Schur  functions (when $f(x)=1-x$) the  action  of $DR_m$ and $DQ_l$ on the generators $h_k$  and $e_k$  ( $k=0,1,2,\dots $)  of algebra symmetric functions $\B$  is
\begin{align}
DR_0(h_k)&=h_k ,\quad
 DR_1(h_k)=-h_{k-1},\quad
 DR_m(h_k)=0, \quad (m=2,3,\dots),\label{e:11}\\
DQ_0(e_k)&=e_k ,\quad
 DQ_1(e_k)=e_{k-1}, \quad
 DQ_m(e_k)=0,\quad (m=2,3,\dots). \label{e:12}
\end{align}
Also  one should  note that in  case of Schur functions  the  operators  $DQ_k$  and $DR_k$  are the  adjoint operators   to the  operators of multiplication   by  complete symmetric  functions $h_k$ and  elementary  symmetric  functions $(-1)^ke_k$ (see \cite{Md}, I- Example 5.3).
\end{remark}

\begin{remark}
A collection of  additive  maps  acting on a ring   satisfying the  properties  of the form  (\ref{e:a2}) are  called  Hasse\,-\,Schmidt's derivations  (see e.g. \cite{GS})

\end{remark}

\subsection{Decomposition of creation and annihilation  operators}\label{decomp}

The action of operators $\Psi^{\pm}_k$ on the coefficients  of   generating functions  $Q(\bar u)$, $R(\bar u)$ can be expressed through multiplication operators and adjoint  operators. We show this  in the  form of  compositions of generating  functions of the corresponding  operators.
%%%%%%%porpos%%%%%%%%
\begin{proposition}\label{DR}
Assume that coefficients $\{Q_{(\lambda_1,\dots,  \lambda_l)}|\, \lambda_i\in\bZ,\quad   l=1,2,\dots,\}$   as  well as   coefficients $\{R_{(\lambda_1,\dots,  \lambda_l)}|\,  \lambda_i\in\bZ,\quad   l=1,2,\dots\}$ span the vector space $\B$ and  that  $\Psi^\pm_k$ are well-defined operators on the linear vector space $\B$. Then
\begin{align}
\Psi^+(v)=Q(v)\circ DR(v),\quad
\Psi^-(v) =R(v)\circ DQ(v).\,\label{pqr}
\end{align}

\end{proposition}
%%%%%%%%55

%\begin{align}
%\Psi^+(v)\,\big(Q(\bar u)\big)&=Q(v)\circ DR(v) \, \big( Q(\bar u)\big)=Q(v,\bar u),\\
%\Psi^-(v) \, \big( R(\bar u)\big)&=R(v)\circ DQ(v)\, \big( R(\bar u)\big)= R(v,\bar u).
%\end{align}
\begin{proof}
Let us prove the first  equality of  (\ref{pqr}), the  second one  follows  by   the similar argument.
By the assumption of the  proposition and by (\ref{psi+})  it is enough  to show  that
$Q(v)\circ DR(v) \, \big( Q(\bar u)\big)=Q(v,\bar u)$,
  We have:
\begin{align*}
&Q(v)\circ DR(v)\, \big( Q(\bar u)\big) = Q(v)\circ DR(v)\, \left( \prod_{i<j} f(u_j/u_i) \prod_i Q(u_i)\right)\\
&\quad =Q(v) \prod_{i<j} f(u_j/u_i)  \prod_i DR(v)\, \big(Q(u_i)\big)
= Q(v) \prod_{i<j} f(u_j/u_i)  \prod_i f(v/u_i) Q(u_i)= Q(v,\bar u).
\end{align*}
\end{proof}
\begin{corollary}
Under the assumptions of Proposition \ref{DR},
\begin{align*}
\Psi^-(v)\, \big(Q(\bar u)\big)&= \prod_{i=1}^l f^{-1}( u_i/v)R(v)Q(\bar u),\\%\label{pq-}\\
\Psi^+(v)\, \big(R(\bar u)\big)&= \prod_{i=1}^l f^{-1}( u_i/v)Q(v)R(\bar u).%\label{pr+}
\end{align*}
\end{corollary}
%\begin{proof}
%\begin{align*}
%\Psi^-(v)\,\big( Q(\bar u)\big) &=R(v)\circ DQ(v)\,\big(Q(\bar u )\big)= R(v)\circ DQ(v) \left(\prod_{i<j} f(u_j/u_i)\prod_i Q(u_i)\right)\\
%&= R(v)\prod_{i<j} f(u_j/u_i) \prod_{i=1}^{l} f^{-1}(u_i/v)\prod_i Q(u_i)= R(v)\prod_{i=1}^{l} f^{-1}(u_i/v) Q(\bar u),
%\end{align*}
%which  proves  (\ref{pq-}).The   statement of  (\ref{pr+}) is  proved similarly.
%\end{proof}

 We will also need the following  equalities implied  by Proposition \ref{DR}:
\begin{align}
\Psi^+(u)\circ \Psi^+(v)&%=Q(u)\circ DR(u)\circ Q(v)\circ DR(v)=
= f(v/u) \, Q(u)\circ Q(v)\circ DR(u)\circ DR(v),\label{r11}
\\
\Psi^-(u)\circ \Psi^-(v)&%=R(u)\circ DQ(u)\circ R(v)\circ DQ(v)
= f(v/u) \, R(u)\circ R(v)\circ DQ(u) \circ DQ(v),\label{r12}
\\
\Psi^+(u)\circ \Psi^-(v)&%=Q(u)\circ DR(u)\circ R(v)\circ DQ(v)=
=f^{-1}(v/u) \, Q(u)\circ R(v)\circ DR(u)\circ DQ(v),\label{r13}
\\
\Psi^-(v)\circ \Psi^+(u)&%=R(v)\circ DQ(v)\circ Q(u)\circ DR(u)=
=f^{-1}(u/v) \, Q(u)\circ R(v)\circ DR(u)\circ DQ(v).\label{r14}
\end{align}
These equalities  represent, in terms of  generating functions, the so-called  normal reordering of the products  of operators and  allow us  to write fermion-like   relations  in  many important examples. Note that   our  approach  will provide relations  for  vertex operators  of fermions action (or their  twisted  analogues) as an easy   consequence of  the definition of  operators $\Psi_k^\pm$.
 %%%%%%%%%%%%%%%%%%%%%%%%%%%%%%%%%%

%%%%%%%%%%%%%%%%%%%%%%%%%%%%%%
\subsection{Action of algebra of fermions: the case of  Schur functions}\label{sec:ferm}

 Recall that  when  $\B$  is  identified with the ring of symmetric  functions,   formulas  (\ref{pq1})    read as an action   on  Schur functions, which form  a linear basis of $\B$:
\begin{align*}
\Psi^+_k (s_\lambda)= s_{(k,\lambda)}, \quad
\Psi^-_{-k} (s_\lambda)=(-1)^k s_{(k,\lambda^\prime)^\prime},
\end{align*}
where $\lambda$  is a partition, and  $\lambda^\prime$  is a conjugate partition, and  it is  { easy to check } that the action of  operators  $\Psi^\pm_k$  is well-defined on $s_\alpha$'s for all  integer vectors $\alpha$ and on all of linear space $\B$.
Properties   (\ref{lowr1}), (\ref{lowr2}) of $s_\alpha$'s   imply %{ \color{blue} (not exactly Clifford algebra) !}
relations  of the  operators $\Psi ^\pm _{k}$:
\begin{align*}
\Psi^+_k\Psi^+_l &+ \Psi^+_{l-1}\Psi^+_{k+1}=0,\\
\Psi^-_k\Psi^-_l &+ \Psi^-_{l+1}\Psi^-_{k-1}=0,\\
\Psi^-_k\Psi^+_l& + \Psi^+_{l}\Psi^-_{k}=\delta_{-k,l}.
\end{align*}

Proposition \ref{propfer}  below  combines these  relations  in  formulas on  generating functions.
Namely,  we  introduce normalized generating  functions of  operators
\begin{align*}
\Psi^{+}(u,m)= u^{m+1}z\Psi^{+}(u),\quad
\Psi^{-}(u,m)= u^{-m+1}z^{-1}\Psi^{-}(u).
\end{align*}
Note that these normalized    generating  functions  $\Psi^{\pm}(u,m)$ are exactly  vertex operators  (\ref{psi2}), (\ref{psi21})  (or  equivalently,   (\ref{psi1}), (\ref{psi11}))  that describe the action of Clifford algebra on Fock  boson space $\oplus_m z^m \B$.
\begin{proposition}\label{propfer}
One has the  following  relations.
\begin{align*}
&\Psi^{+}(u,m+1)\circ\Psi^{+}(v,m)+\Psi^{+}(v,m+1)\circ\Psi^{+}(u,m)= 0,\\
&\Psi^{-}(u,m-1)\circ\Psi^{-}(v,m)+ \Psi^{-}(v,m-1)\circ\Psi^{-}(u,m)=0,\\
&\Psi^{-} (u,m+1)\circ\Psi^{+}(v,m)+ \Psi^{+}(v,m-1)\circ\Psi^{-}(u,m)=\delta(u^{-1},v^{-1}).
\end{align*}
Here
$$
\delta(u,v)= \sum_{k\in \bZ} \frac{u^k}{v^{k+1}}=  i_{u,v}\left(\frac{1}{v-u}\right)+ i_{v,u}\left(\frac{1}{u-{v}}\right)
$$
is  the formal delta-function distribution and
the notation
$
i_{v,u} g(u,v)
$
 stands  for the expansion of the rational function  $ g(u,v)$  as a formal series expanded in the region $|v|>|u|$. %$u^{-1}$ and $v$.
\end{proposition}
\begin{proof} One can deduce the statement  directly from the relations on the coefficients $\Psi^\pm_k$.  However, we would like to note that formulas  (\ref{r11} - \ref{r14})
provide  even a simpler  proof which can be effortlessly  extended to  the cases of  other  generating  functions  (as it will be illustrated below).
For the  first relation from (\ref{r11})   write  immediately
\begin{align*}
&\Psi^{+}(u,m+1)\circ\Psi^{+}(v,m)+\Psi^{+}(v,m+1)\circ\Psi^{+}(u,m)= \\
&= u^{m+1}v^{m+1}\left(u (1-v/u) +v(1-u/v) \right) \, Q(u)\circ Q(v)\circ DR(u)\circ DR(v)
= 0,
\end{align*}
 and  similarly follows the  second  relation of the proposition.
% {\color{blue}  used that $DR(u)\circ DR(v)=DR(v)\circ DR(u)$ }
Finally,
\begin{align*}
&\Psi^{-} (u,m+1)\circ\Psi^{+}(v,m)+ \Psi^{+}(v,m-1)\circ\Psi^{-}(u,m)\\
&=u^{-m}v^{m} ( v \,i_{u,v}(1-v/u)^{-1} + u\, i_{v,u}(1-u/v)^{-1}) \, Q(v)\circ R(u)\circ DR(v)\circ DQ(u)\\
&=u^{-m}v^{m} \left( \sum_{r=0}^{\infty}  \frac{v^{r+1}}{u^{r}}   +  \sum_{r=0}^{\infty}  \frac{u^{r+1}}{v^{r}} \right) Q(v)\circ R(u)\circ DR(v)\circ DQ(u)\\
&=u^{-m}v^{m} \delta(u^{-1},v^{-1})  Q(v)\circ R(u)\circ DR(v)\circ DQ(u)= \delta(u^{-1},v^{-1})\cdot Id
\end{align*}
\end{proof}

%%%%%%%%%%%%%%%%%%%%%%%%%%%%%%%%%%%%
\subsection{Twisted   fermions  action}
More generally,  consider  a   generating function (\ref{e:Qdef}) with a correlation  function  of the form
$$
f(x)= \frac{1-x}{p(x)},
$$
where $p(x)$  is a polynomial in $x$.
Then  relations  (\ref{r11}) - (\ref{r14}) immediately  imply   ``twisted''   fermion relations  for $\Psi^\pm (u)$:
similar to Proposition \ref{propfer} calculations show that
\begin{align*}
u\, p(v/u) \Psi^{\pm}(u)\Psi^{\pm}(v)+v\, p(u/v) \Psi^{\pm}(v)\Psi^{\pm}(u)&= 0,\\
%\end{align*}
%and
%\begin{align*}
 vp(u/v) \Psi^{-} (u)\Psi^{+}(v)+  up(v/u) \Psi^{+}(v)\Psi^{-}(u)
&=p(1)^2 \delta(u^{-1},v^{-1})\cdot Id.
\end{align*}
In  particular,  if
$p(x) = 1-tx$,   the  series
\begin{align}
Q(\bar u)&:=\prod_{i<j} \frac{u_i-u_j}{u_i-tu_j} \prod_{i=1}^{l} Q(u_i)\label{fHL}
\end{align}
is the generating  function for Hall-Littlewood  functions (see  Section \ref{HL} for a short review), and
we obtain the relations
for vertex  algebra realization for Hall-Littlewood  functions, which  were  proved  earlier by different  methods   in \cite{J12}:
\begin{align*}
&(u-tv) \Psi^{\pm}(u)\Psi^{\pm}(v)+(v-tu) \Psi^{\pm}(v)\Psi^{\pm}(u)= 0,\\
&(v-tu) \Psi^{-} (u)\Psi^{+}(v)+  (u-tv) \Psi^{+}(v)\Psi^{-}(u)=(1-t)^2 \delta(u^{-1},v^{-1})\cdot Id
\end{align*}
%%%%%%%%%
%%%%%%%%%
%%%%%%%%%
%%%%%%%%%
%%%%%%%%%

\section{Shifted  symmetric functions}\label{Shift}
In this  section we apply the techniques described above to  deduce a formula for the   (shifted) generating  function of  shifted  Schur  functions and  their    realization analogues  to vertex operators  presentation of  classical symmetric  functions.
The necessary  information on shifted  symmetric functions as well as   some details on their role in representation theory  are
provided in  Section \ref{ashift}.
\subsection{Shifted generating  functions}
Let $f(u)$  be a formal series     or a  function in variable $u$ in some  general sense.
We introduce  the {\it  shift operator}
$$
e^{k\partial_u}\, ( f(u) )= f(u+k).
$$
This  exponential notation is  motivated  by  Taylor  series expansion  formula, where   for an appropriate  class of  functions   in  an  appropriate   domain of convergence one can write
 $$
 f(u+k )=\sum_{s=0}^{\infty}\frac{\, ( k\,\partial_u)^s }{s!} \, \big( \,f(u)\,\big)= e^{k\, \partial_u} \,\big(f(u)\big).
 $$
We will use   short notation
$
e^{k\,\partial_i}:=e^{k\,\partial_{u_i}}
$
for shifts  along  variable  $u_i$ acting on   $f(u_1,\dots, u_l)$.

A falling  factorial  power of $u$  is defined by
\begin{align}
 (u|k)=
  \begin{cases}
 u(u-1)\cdots (u-k+1) \quad \text{for  $k=1,2\dots,$}\\
 1,\quad \text{for  $k=0$,}\\
  \frac{1}{(u+1)\cdots (u+(-k) )} \quad \text{for  $k=-1,-2\dots.$}\\
 \end{cases}\label{falf}
 \end{align}
 Note that   the shifted $k$-th power sum is a result of application  to constant  function $1$ of  the $k$-th power of the  following  operator:
 $$
 \left({u}e^{-\partial_u}\right)^k \,(\,1\,)= {(u|k)}.
 $$
 We will  be interested in  shifted  generating  functions, which will be  infinite sums in monomials of  shifted  powers of  formal variables $u_i$'s.

Let $\B^*$  be the  ring of shifted symmetric functions.
Let  $ h^*_k$  be the shifted  complete symmetric functions and  $e^*_k$ the shifted elementary symmetric functions defined by (\ref{h^*}), (\ref{e^*}).
Consider the  shifted generating  functions \footnote{$Q^*(u)$ and $R^*(u)$  in this  note  correspond to    $H^*(u)$ and $E^*(-u-1)$ respectively  in \cite{OO1}.}
\begin{align}
Q^*(u)=\sum_{k= 0}^{\infty} \frac{h^*_k}{(u|k)},\quad
R^*(u)=\sum_{k=0}^{\infty} {(-1)^ke^*_k}{(u|-k)}.\label {qr*}
\end{align}

It is proved in \cite {OO1} (Corollary 12.3) that
$$
Q^*(u) R^*(u)=1.
$$

Also  consider the   formal  series    of the  following  form:
\begin{align*}
YQ^*(u)=\sum_{k\ge 0}  h^*_k\left(\frac{1}{u} e^{-\partial_u}\right)^k,\quad\quad
YR^*(u)=\sum_{k\ge 0} (-1)^k e^*_k\left( e^{\partial_u}  \frac{1}{u} \right)^k,
\end{align*}
where   $h^*_k$ and $e^*_k$  are viewed as the multiplication operators by respective functions acting on  the space  $\B^*$.
Then (\ref{qr*}) can be  interpreted  as a result of  application  of these ``operators''  to the vacuum vector:
\begin{align*}
Q^*(u)= YQ^*(u)\, (1),\quad R^*(u)= YR^*(u)\,(1).
\end{align*}

As before,  we use the short notation for  multivariable arguments of functions:
$$\bar u=(u_1,\dots, u_l).$$

For a matrix  $A= (a_{ij})_{ij=1,\dots, N}$ with  non-commutative entries,
the determinant is defined by
$$
\det (A) = \sum_{\sigma \in S_N} sgn(\sigma) a_{1\sigma(1)}\cdots a_{N\sigma(N)}.
$$

Then one defines that
\begin{align}\label{YQ1}%\label{YQ1}
YQ^*( \bar u)&= \det\left[   \left( \frac{1}{u_i} e^{-\partial_i}\right)^{i-j}  e^{(1-j)\partial_i}\right]\circ \prod_{i=1}^{l} YQ^*(u_i),\\
YR^*(\bar u)&= \det\left[  \left( e^{\partial_i}\frac{1}{u_i} \right)^{i-j}   e^{(j-1)\partial_i}\right]\circ \prod_{i=1}^{l} YR^*(u_i),  \label{YR1}%\label{YR1}
\end{align}
and  the result of  application of  (\ref{YQ1}) or  (\ref{YR1})   to the constant function $1$  is a formal   series in   shifted  powers  of $u$  with   coefficients in $\B^*$, which we denote as
\begin{align}\label{Q1}%\label{Q1}
Q^*( \bar  u)= YQ^*(\bar u) ( 1),\\ R^*( \bar u)= YR^*( \bar u) (1).
\end{align}

\begin{proposition}
One has that
\begin{align}\label{Q3} %\label{Q3}
 Q^*(\bar u)&=\det\left[ \frac{1}{(u_i| i-j)}\right]\prod_{i=1}^{l} Q^*(u_i-i+1),
\\
R^*(\bar u)&=\det\left[ {(u_i|j-i)} \right]  \prod_{i=1}^{l} R^*(u_i+i-1).\label{R1}
\end{align}

Also
\begin{align}\label{Q2}%\label{Q2}
Q^*( \bar u)&= \prod _{i=1}^{l} \left(\frac{1}{u_i}e^{-\partial_i}\right)^{i-1} \left(\prod_{i<j} (u_j-u_i) \,\prod_{i=1}^{l} Q^*(u_i)\right),\\
R^*( \bar u)&= \prod _{i=1}^{l} \left(e^{\partial_i} \frac{1}{u_i}\right)^{i-1}\left(\prod_{i<j} (u_j-u_i)\,\prod_{i=1}^{l} R^*(u_i) \right).\label{R2}%\label{R2}
\end{align}

\end{proposition}

\begin{proof}
Since for any $k\in\bZ$,
\begin{align*}
 \left( \frac{1}{u} e^{-\partial_u}\right)^{k}  =\frac{1}{(u|k)} e^{{-k}\partial_u},
\end{align*}
one has
\begin{align*}
YQ^*( \bar u )
&=\det\left[ \frac{1}{(u_i|i-j)} \right]\prod_{i=1}^{l}e^{({1-i})\partial_i}\circ YQ^*(u_i),
\end{align*}
and similarly,
\begin{align*}
YR^*( \bar  u)=&\det\left[{(u_i|j-i)} \right]\prod_{i=1}^{l}e^{({i-1})\partial_i}\circ YR^*(u_i).
\end{align*}
Then (\ref{Q3}) and (\ref{R1})  follow by application of  $YQ^*( \bar u )$ and $YR^*( \bar u )$ to the vacuum vector  $1$.

For a proof of  (\ref{Q2})  and (\ref{R2}), we see that
\begin{align*}%\label{Q2}
 \prod _{i=1}^{l} \left(\frac{1}{u_i}e^{-\partial_i}\right)^{i-1}\left(\prod_{i=1}^{l} Q^*(u_i)\prod_{i<j} (u_j-u_i)\right)=
\prod_{i=1}^{l} Q^*(u_i-i+1)  \prod _{i=1}^{l} \left(\frac{1}{u_i}e^{-\partial_i}\right)^{i-1}\left(\prod_{i<j} (u_j-u_i) \right).
\end{align*}
\begin{lemma}\label{lem2}
\begin{align*}
\prod_{i=1}^{l}\left(\frac{1}{u_i}e^{-\partial_i}\right)^{i-1} \left (\prod_{i<j} (u_j-u_i)\right) &=\det\left[ \frac{1}{(u_i|i-j)} \right],\\
\prod_{i=1}^{l}\left( e^{\partial_i}\frac{1}{u_i}\right)^{i-1}\left(\prod_{i<j} (u_j-u_i)\right)&=\det\left[ {(u_i|j-i)} \right] .
\end{align*}
\end{lemma}
\begin{proof}
We  can  rewrite the first product  using  the Vandermonde determinant:
\begin{align}
&\prod_{i=1}^{l}\left(\frac{1}{u_i}e^{-\partial_i}\right)^{i-1}\left(\prod_{i<j} (u_j-u_i)\right)
=\prod_{i=1}^{l}\frac{1}{(u_i|i-1)}\prod_{i<j} (u_j -j-u_i +i)\notag \\
&=\prod_{i=1}^{l}\frac{1}{(u_i|i-1)}\prod_{i<j} ((u_j-j+2 )-(u_i-i+2))
=\prod_{i=1}^{l}\frac{1}{(u_i|i-1)}\det\left[ (u_i-i+2)^{j-1} \right].\label{mat1}%\label{mat1}
\end{align}
Recall  that  shifted and  ordinary powers  of a variable $x$ are related by
$$
x^m= \sum_{k=0}^{m} (-1)^{m-k}S(m,k)  x(x+1)\cdots (x+k-1),
$$
where $S(m,k)$  are  the Stirling  numbers of the second kind, and $S(m,m)=1$.
Therefore, with $x= u_i-i+2$  we can
expand  by linearity the columns of the  determinant
  to get  (\ref{mat1})    equal to
\begin{align*}
 \prod_{i=1}^{l}\frac{1}{(u_i|i-1)}\det\left[ (u_i-i+2) (u_i-i+3)\cdots (u_i-i+j) \right]
&= \det\left[ \frac{1}{(u_i|i-j)} \right] .\notag
\end{align*}

Similarly,
\begin{align*}
\prod_{i=1}^{l}\left( e^{\partial_i}\frac{1}{u_i}\right)^{i-1}\left(\prod_{i<j} (u_j-u_i)\right)
&=\prod_{i=1}^{l} (u_i|1-i)\det [(u_i+i-1)^{j-1}] \\
&=\prod_{i=1}^{l} (u_i|1-i)\det [(u_i+i-1|j-1)] = \det [(u_i|j-i)] .
\end{align*}
\end{proof}
Then (\ref{Q2})  and (\ref{R2})  are  implied  by   Lemma \ref{lem2}.
This completes the  proof of the proposition.
\end{proof}

\begin{remark}
Note from (\ref{Q2}), (\ref{R2}) that the change of order of  variables  leads to the  following  equalities:
\begin{align*}
\frac{1}{u}e^{-\partial_u}\left(Q^*(u,v, u_2,u_3,\dots)\right)&= -\frac{1}{v}e^{-\partial_v}\left(Q^*(v,u, u_2,u_3,\dots)\right),\\
e^{\partial_u}\frac{1}{u}\left(R^*(u,v, u_2,u_3,\dots)\right)&= -e^{\partial_v}\frac{1}{v}\left(R^*(v,u, u_2,u_3,\dots)\right).
\end{align*}
\end{remark}
\subsection{Shift automorphisms  on $\B^*$}
Following \cite{OO1}, Section  13, observe that the formal action of the shift  operator $e^{-\partial_u} $ on  the  generating function  $Q^*(u)$ corresponds to  an automorphism  $\tau: \B^*\to \B^*$ of shifted symmetric functions. Namely,  write  for $a\in \bZ_{\ge 0}$,
\begin{align}
Q^*(u-a)= e^{-a\partial_u} Q^*(u)= \sum_{k=0}^{\infty}\frac{h^*_k }{(u-a|k)}= \sum_{k=0}^{\infty}\frac{\tau^a(h^*_k) }{(u|k)}.\label{e:tau}%\label{e:tau}
\end{align}

Note that for $ k=1,2,\dots$,
\begin{align}
\frac{1}{(u-1|k)}&=\frac{1}{(u|k)}+\frac{k}{(u|k+1)},\notag \\
{(u+1|-k)}&={(u|-k)}- k{(u|-k-1)}. \quad \label{sh1}
\end{align}
Hence, the  explicit action of $\tau$  on the generators $h^*_k$  ($k=1,2,\dots$) is given by
\begin{align}
\tau(h^*_k)&= h^*_k+(k-1) h^*_{k-1}, \quad \label{tau1} \\
\tau^a(h^*_k)&=\sum_{i=0}^{a}{a\choose i} (k-1|i) h^*_{k-i}   \quad  ( a =1,2,\dots).\label{tau11}
\end{align}

Similarly,   %c.f. \cite{OO1}  (13.7))
for  $k=1,2,\dots$,
\begin{align}
\tau^{-1}(e^*_k)&= e^*_k+(k-1) e^*_{k-1}, \quad  \label{tau-1} \\
\tau^{-a}(e^*_k)&= \sum_{i=0}^{a}{a\choose i} (k-1|i) e^*_{k-i}, \quad  ( a =1,2,\dots),\label{tau-11}
\end{align}
which by (\ref{sh1})  corresponds to a shift of a variable of  the generating function   $R^*(u)$ for $a\in \bZ_{\ge 0}$:
$$
e^{a\partial_u} R^*(u)= R^*(u+a)= \sum_k(-1)^{k} e^*_k (u+a|-k)=  \sum_k(-1)^{k} \tau^{-a}(e^*_k )(u|-k).
$$

\subsection{Coefficients of shifted  generating  functions}
Our next goal  is  to identify the coefficients  $Q^*_\lambda$ and  $R^*_\lambda$  of the (shifted)  expansion of  the generating functions $Q^*(\bar u)$   and $R^*(\bar u)$   with the shifted Schur functions. We need the following statement.
\begin{lemma}\label{lem_b}
Let $\tau ^{\pm 1}$ be  the automorphisms of $\B^*$  defined by (\ref{tau1}), (\ref{tau-1}), and let $k_i\in \bZ_{\ge 0}, m_i\in \bZ $ ($i=1,\dots,l$).Then
in the (shifted)  expansions
\begin{align*}
&\prod_{i=1}^{l}\frac{1}{(u_i|m_i)}Q^*(u_i-k_i-m_i) =\sum_{\lambda} C_\lambda {\frac{1}{(u_1|\lambda_1)\cdots (u_l|\lambda_l)}},
\\
&\prod_{i=1}^{l}{(u_i|-m_i)}R^*(u_i+k_i+m_i) =\sum_{\lambda} D_\lambda {{(u_1|-\lambda_1)\cdots (u_l|-\lambda_l)}},
\end{align*}
where the coefficient $C_{\lambda}$
 in the  first expansion is the monomial given by
$$C_\lambda= \tau^{k_1}(h^*_{\lambda_1-m_1})\cdots\tau^{k_l}(h^*_{\lambda_l-m_l}),$$ and
 the coefficient  $D_{\lambda}$  % of the term ${(u_1|-\lambda_1)\dots (u_l|-\lambda_l)}$
 in the second expansion is the monomial given by
$$D_\lambda= (-1)^{(\sum{\lambda_i}-\sum m_i)}\, \tau^{-k_1}(e^*_{\lambda_1-m_1})\cdots\tau^{-k_l}(e^*_{\lambda_l-m_l}).$$
\end{lemma}
\begin{proof}
The statement  is implied  by the  following  argument for  $k\in \bZ_{\ge 0}, m\in \bZ $ :
\begin{align*}
\frac{1}{(u|m)}Q^*(u-k-m)&=\sum_{p} \frac{\tau^k(h^*_p)}{(u|m)(u-m|p)}= \sum_{p} \frac{\tau^k(h^*_p)}{(u|m+p)}=\sum_{a} \frac{\tau^k(h^*_{a-m})}{(u|a)},\\
{(u|-m)}R^*(u+k+m)&=\sum_{p} {(-1)^p\tau^{-k}(e^*_p)}{(u|-m)(u+m|-p)}\\ &= \sum_{p} {(-1)^p\tau^{-k}(e^*_p)}{(u|-m-p)}=\sum_{a} {(-1)^{a-m}\tau^{-k}(e^*_{a-m})}{(u|-a)}.
\end{align*}
\end{proof}
The  following proposition states that $Q^*(\bar u)$  and $R^*(\bar u)$ are generating  functions for the shifted Schur  functions.
\begin{proposition}\label{gen*}
Let $\lambda$ be an integer  vector with at most $l$  non-zero parts,  let $N=\sum_{i}\lambda_i$, and  let  $\lambda^\prime$  be a conjugate vector.
The  coefficient of
${\frac{1}{(u_1|\lambda_1)\cdots (u_l|\lambda_l)}}$ in  a shifted expansion of
 $Q^*( u_1,\dots, u_l)$  is  a
shifted Schur  function $s^*_\lambda$, and
the  coefficient of
${{(u|-\lambda_1)\cdots (u|-\lambda_l)}}$ in
 $R^*( u_1,\dots, u_l)$  is  a
shifted Schur  function $ (-1)^Ns^*_{\lambda^\prime}$:
\begin{align*}
Q^*( u_1,\dots, u_l) &=\sum_{l(\lambda)\le l}\frac{s^*_\lambda}{(u_1|\lambda_1)\cdots (u_l|\lambda_l)},\\
R^*( u_1,\dots, u_l)&=\sum_{l(\lambda)\le l} (-1)^{N} {s^*_{\lambda^\prime}}{(u_1|-\lambda_1)\cdots (u_l|-\lambda_l)}.
\end{align*}
\end{proposition}

\begin{proof} The expansion  of determinant  (\ref{Q3}) gives
 \begin{align*}
 Q^*( u_1,\dots, u_l)=\sum_{\sigma\in S_l} sgn(\sigma)  \frac{1}{(u_1| 1-\sigma(1))}\cdots \frac{1}{(u_l| l-\sigma(l))} \prod_{i=1}^{l} Q^*(u_i- i+1)\\
=\sum_{\sigma\in S_m} sgn(\sigma)  \frac{1}{(u_1| 1-\sigma(1))}\cdots \frac{1}{(u_l| l-\sigma(l))} \prod_{i=1}^{l} Q^*(u_i- (i-\sigma(i))- (\sigma(i)-1))\\
 \end{align*}
Set $k_i= \sigma(i)-1$   and  $m_i=i-\sigma(i)$ for $i=1,\ldots, l$. Observe that  $k_i\ge 0$ for any $i=1,\dots, l$, so by Lemma \ref{lem_b}, we obtain the coefficient of ${\frac{1}{(u_1|\lambda_1)\cdots (u_l|\lambda_l)}} $ is
$$
Q^*_\lambda
=\sum_{\sigma\in S_l} sgn(\sigma) \tau^{\sigma(1)-1} (h^*_{\lambda_1-1+\sigma(1)})\, \cdots\,  \tau^{\sigma(l)-1}(h^*_{\lambda_l-l +\sigma(l)})
=\det [\tau^{j-1}h^*_{\lambda_i-i+j}].
$$
Recall that  shifted  Schur  functions  satisfy  the Jacobi\,--\,Trudi  identity (\ref{shiftJT2}).
Exactly the same  argument   as in Section  \ref{Sch1} allows us to  use  the  Jacobi\,--\,Trudi  identity to
  extend the definition of   shifted  Schur functions  to  integer vectors $\alpha= (\alpha_1,\dots, \alpha_l)$, thus we have that
$
s^*_\alpha = sgn(\sigma) s^*_{\lambda},
$
if $\alpha -\rho_l=\sigma(\lambda -\rho_l)$ for some  $\sigma \in S_l$  and a partition $\lambda$, and $s^*_\alpha =0$ otherwise.

 Hence, the coefficient $Q^*_\lambda$ can be identified  with  shifted Schur function $s^*_\lambda$.
 Similarly,

 \begin{align*}
 R&^*( u_1,\dots, u_l)=\sum_{\sigma\in S_m} sgn(\sigma)  {(u_1| \sigma(1)-1)}\cdots {(u_l| \sigma(l)-l)} \prod_{i=1}^{l} R^*(u_i+i-1)\\
&=\sum_{\sigma\in S_m}  sgn(\sigma)  {(u_1| \sigma(1)-1)}\cdots {(u_l| \sigma(l)-l)}  \prod_{i=1}^{l} R^*(u_i +(i-\sigma(i))+ (\sigma(i)-1)).
 \end{align*}
Set $k_i= \sigma(i)-1$   and  $m_i=i-\sigma(i)$ for $i=1,\dots, l$. Then  $k_i\ge 0$ for any $i=1,\dots, l$,  $\sum_i m_i=0$, so by Lemma \ref{lem_b} and (\ref{shiftJT2}),
 \begin{align*}
R^*_\lambda
&=(-1)^{\sum\lambda_i}\sum_{\sigma\in S_l} sgn(\sigma) \tau^{1-\sigma(1)} (e^*_{\lambda_1-1+\sigma(1)})\, \cdots\,  \tau^{1-\sigma(l)}(e^*_{\lambda_l-l +\sigma(l)})\\
&=(-1)^{N}\det [\tau^{1-j}e^*_{\lambda_i-i+j}] = (-1)^{N} s^*_{\lambda^\prime}.
 \end{align*}
This completes the proof of the proposition.

\end{proof}

%%%%%%%%%%%%%%%%%%%%%%%%%%%%%
\subsection{Creation and  annihilation operators for shifted Schur functions}  %%%%%%%%%%%%%%%%%%%%%%
Similarly to   Section \ref{sec:ferm}, we introduce  creation and  annihilation operators $\Psi^{\pm}_k$, $k\in \bZ$,  acting on the  set % (linear span?)
of coefficients  $Q^*_\lambda$'s
 ($R^*_\lambda$'s) by
\begin{align}
\Psi^+_k (Q^*_\lambda)= Q^*_{(k,\lambda)}, \quad
\Psi^-_{-k} (R^*_\lambda)= R^*_{(k,\lambda)}.\label{Pp*}
\end{align}
which  read  by Proposition \ref {gen*} as
\begin{align}
\Psi^+_k (s^*_\lambda)= s^*_{(k,\lambda)}, \quad
\Psi^-_{-k} (s^*_\lambda)=(-1)^k s^*_{(k,\lambda^\prime)^\prime}.\label{Pp2*}
\end{align}
Shifted  Schur functions $s^*_\lambda$   span  the space of  shifted symmetric  functions   $\B^*$ and   (\ref{Pp2*})  define  how linear  operators   $\Psi^{\pm}_k $  act on  all of $\B^*$.

Recall that  shifted  Schur  functions  enjoy  exactly the same  lowering-raising   properties as classical  Schur  functions:
\begin{align*}
 s^*_{(\dots, \alpha_i, \alpha_{i+1},\dots)}= - s^*_{(\dots, \,\alpha_{i+1} -1\,, \,\alpha_{i}+1\, ,\dots)},
\end{align*}
which   imply  exactly the same  commutation relations  of operators  $\Psi^{\pm}_k $    as in the classical case:
\begin{align*}
\Psi^+_k\Psi^+_l &+ \Psi^+_{l-1}\Psi^+_{k+1}=0,\\
\Psi^-_k\Psi^-_l &+ \Psi^-_{l+1}\Psi^-_{k-1}=0,\\
\Psi^-_k\Psi^+_l& + \Psi^+_{l}\Psi^-_{k}=\delta_{-k,l}.
\end{align*}
Let us  rewrite these relations  in terms of  (shifted)  generating  functions in the spirit of vertex operator formalism.
Define
\begin{align*}
\Psi^+(v)=\sum_{k\in \bZ}\frac{\Psi^+_k}{ (v|k)}, \quad
\Psi^-(v)=\sum_{k\in \bZ}\Psi^-_{k} (v|-k).
\end{align*}
Then
\begin{align}
\Psi^+(v)\left(Q^*(\bar u)\right)= Q^*(v,u_1,\dots u_l),\quad
\Psi^-(v)\left(R^*(\bar u)\right)= R^*(v,u_1,\dots u_l),\label{p-r*}
\end{align}
and generating functions  of shifted Schur functions can be  viewed as a result of  application of
$\Psi^\pm(v)$  to vacuum vector:
\begin{align*}
Q^*(u_1,\dots u_l)= \Psi^+(u_1)\circ\dots\circ  \Psi^+(u_l) \,(1),\\
R^*(u_1,\dots u_l)= \Psi^-(u_1)\circ \dots\circ  \Psi^-(u_l) \, (1).
\end{align*}
The  commutation relations are
\begin{align}
\frac{1}{u}e^{-\partial_u}\circ\Psi^+(u)\circ\Psi^+(v)&+\frac{1}{v}e^{-\partial_v}\circ\Psi^+(v)\circ\Psi^+(u)=0,%\label{ppp1}
\\
e^{\partial_u}\circ\frac{1}{u}\Psi^-(u)\circ\Psi^-(v)&+e^{\partial_v}\circ\frac{1}{v}\Psi^-(v)\circ\Psi^-(u)=0,%\label{ppp2}
\\
\Psi^+(u)\circ\Psi^-(v)& +\Psi^-(v)\circ\Psi^+(u)=\sum_{k\in \bZ}\frac{(u|k)}{(v|k)}\cdot Id.%\label{ppp3}
\end{align}

\subsection{Adjoint operators  for shifted symmetric functions}
Finally, we want to  write  $\Psi^\pm(u)$ in  ``normally ordered form''  similarly to   (\ref{pqr}).
 We define operators $DR^*_k$, $DQ^*_k$  acting on $\B^*$ as follows.
Let
 \begin{align*}
 DR^*(u)= \sum_{m=0}^{\infty } DR^*_m (u|m),\quad
 DQ^*(u)= \sum_{m=0}^{\infty } DQ^*_m \frac{1}{(u|-m)},
 \end{align*}
be the formal  shifted series with  $DR^*_m,  DQ^*_m$  being   linear operators    acting on $\B^*$ and  such that the series $DR^*(u)$  and $ DQ^*(u)$ have  the property
\begin{align}
DR^*(u) (Q^*(v))&=\frac{1}{v} e^{-\partial_v}\left((v-u)Q^*(v)\right)
= \left(1-\frac{u+1}{v}\right)Q^*(v-1),\label{DR*1}\\
DQ^*(u) (R^*(v))&=e^{\partial_v}\frac{1}{v} \left((v-u)R^*(v)\right)
= \left(1-\frac{u}{v+1}\right)R^*(v+1).
\label{DR*}
\end{align}
Formula (\ref{DR*}) describes the action  of $DR^*_m,  DQ^*_m$ on generators $h^*_k$ and $e^*_k$    respectively. For example,   $DR^*_m(h^*_k)$  is the coefficient of $\frac{(u|m)}{(v|k)}$
in the  expansion  of the  first equation of (\ref{DR*}):
\begin{align*}
 &Q^*(v-1)- \frac{u+1}{v} Q^*(v-1)=
 \sum_{k=0}^{\infty}\frac{\tau (h^*_k) }{(v|k)}
- (u+1) \sum_{k= 0}^{\infty} \frac{h^*_{k-1}}{(v|k)}.
\end{align*}
Therefore, for $k=0, 1,2,\dots$,
\begin{align*}
DR^*_0(h^*_k)&=\tau(h^*_k)- h^*_{k-1}=h^*_k +(k-2)h^*_{k-1},\quad \\
 DR^*_1(h^*_k)&=-h^*_{k-1},\quad
 DR^*_m(h^*_k)=0\quad ( m=2,3,\dots).
\end{align*}

Similarly,
 $(-1)^kDQ^*_m(e^*_k)$  is the coefficient of $\frac{(v|-k)}{(u|-m)}$
in the  expansion  of the  second equation of (\ref{DR*}), which  gives for $k=0, 1,2,\dots$
\begin{align*}
DQ^*_0(e^*_k)&=\tau^{-1}(e^*_k)- e^*_{k-1}=e^*_k +(k-2)e^*_{k-1},\\
 DQ^*_1(e^*_k)&=e^*_{k-1}, \quad
 DQ^*_m(e^*_k)=0 \quad (m=2,3,\dots).
\end{align*}

\begin{remark}
One can compare these formulas with (\ref{e:11}) and (\ref{e:12}).
\end{remark}

Next,  we extend  the action of $DR^*_m$, $DQ^*_m$  to all of $\B^*$,  by linearity and  by  the  rule
\begin{align}
DR^*(u)\left(\prod_{i=1}^l Q^*(u_i)\right)=\prod_{i=1}^l DR^*(u)(Q^*(u_i)),\label{mon1}\\
DQ^*(u)\left(\prod_{i=1}^l R^*(u_i)\right)=\prod_{i=1}^l DQ^*(u)(R^*(u_i)).\label{mon2}
\end{align}
In particular, write
\begin{align*}
\prod_{i=1}^l   Q^*(u_i)&=\sum_{0\le \lambda_1,\dots,\lambda_l<\infty }
\frac {h^*_{\lambda_1}\cdots h^*_{\lambda_l} } {(u_1|\lambda_1)\cdots (u_l|\lambda_l)},\\
\prod_{i=1}^l  R^*(u_i)&=\sum_{0\le \lambda_1,\dots,\lambda_l<\infty }
(-1)^{|\lambda|} e^*_{\lambda_1} \cdots {e^*_{\lambda_l} } {(u_1|-\lambda_1)}\cdots  {(u_l|-\lambda_l)},
\end{align*}
(here $|\lambda|=\sum\lambda_i$) to get the action of $DR^*(u)$ on the monomials that  span $\B^*$
\begin{align*}
DR^*(u)(h^*_{\lambda_1}\dots h^*_{\lambda_l})=DR^*(u)(h^*_{\lambda_1})\cdots DR^*(u)(h^*_{\lambda_l})=\prod_{i=1}^{l}(h_{\lambda_i} +(\lambda_i-2)h_{\lambda_i-1}- h_{\lambda_i} u),\\
DQ^*(u)(e^*_{\lambda_1}\dots e^*_{\lambda_l})=DR^*(u)(e^*_{\lambda_1})\cdots DR^*(u)(e^*_{\lambda_l})=\prod_{i=1}^{l}(e_{\lambda_i} +(\lambda_i-2)e_{\lambda_i-1}+ e_{\lambda_i} u),
\end{align*}
and expand these products in shifted powers of $u$ to get the explicit  values  of $DR^*_m(h^*_{\lambda_1}\cdots h^*_{\lambda_l})$ and $DQ^*_m(e^*_{\lambda_1}\cdots e^*_{\lambda_l})$.
For example,
\begin{align*}
DR^*(u)(h^*_a h^*_{b})=&(h_{a} +(a-2)h_{a-1})(h_{b} +(b-2)h_{b-1})\\
&-u(h_{a}h_{b}  +(a-2)h_{b}h_{a-1}+(b-2)h_{a}h_{b-1})
+(u|2)h_{a}h_{b},
\end{align*}
Hence,
\begin{align*}
DR^*_0(h^*_a h^*_{b})&=(h_{a} +(a-2)h_{a-1})(h_{b} +(b-2)h_{b-1}),\\
DR^*_1(h^*_a h^*_{b})&=-(h_{a}h_{b}  +(a-2)h_{b}h_{a-1}+(b-2)h_{a}h_{b-1}),\\
DR^*_2(h^*_a h^*_{b})&=h_{a}h_{b},\quad
DR^*_m(h^*_a h^*_{b})=0, \quad (m=3,4,\dots).
\end{align*}
Finally, from (\ref{mon1}), (\ref{mon2}) we write the action of $DR^*(v)$ on $ Q^*(\bar u )$ and of  $DQ^*(v)$ on $ R^*(\bar u)$.
 Formulas (\ref{DR*1}),  (\ref{DR*}),  (\ref{Q2}), (\ref{R2}) immediately  imply
\begin{align}
Q^*(v)\circ DR^*(v)\big( Q^*(u_1,\dots, u_l))= Q^*(v, u_1,\dots, u_l),\label{qz*1}\\
R^*(v)\circ DQ^*(v)\big( R^*(u_1,\dots, u_l))= R^*(v, u_1,\dots, u_l),
\label{qz*2}
\end{align}
and from   (\ref{qz*1}), (\ref{qz*2}) and (\ref{p-r*}) follows  the ``normally  ordered" presentation of $\Psi^\pm(u)$:
\begin{proposition}
\begin{align*}
\Psi^+(v)=Q^*(v)\circ DR^*(v),\quad \Psi^-(v)=R^*(v)\circ DQ^*(v).
\end{align*}
\end{proposition}

%%%%%%%%%%%%%
\section{Appendix}\label{ap}
In this  section,   we  collect  some facts  and basic materials %references
used in the exposition above for the convenience of the reader.
\subsection{Boson-fermion  correspondence }\label{BF}
Consider  an infinite-dimensional complex vector space $V=\oplus _{j\in \bZ}\bC\, v_j$    with a linear  basis $\{v_j\}_{j\in \bZ}$.
We  define $F^{(m)}$ ($m\in \bZ$)   as  a linear span of semi-infinite monomials
$v_{i_m}\wedge v_{i_{m-1}}\wedge\dots  $ with the properties: \\
(1)\quad  $i_m>i_{m-1}>\dots$,  \\
(2)\quad  $i_k=k$  for $k<<0$.\\
The monomial   of the form $|m\>=v_m\wedge v_{m-1}\wedge\dots $ is called  the $m$th vacuum vector.
The elements of $F^{(m)}$ are linear  combinations of monomials  $ v_I= v_{i_1}\wedge v_{i_2}\wedge\dots $ that are different from $| m\>$ only at a finitely places, and $I=\{i_1, i_2, \ldots \}$.
The fermionic Fock space is  defined to be the graded space $$ \Fe=\oplus_{m\in \bZ} F^{(m)}.$$
 Many  important  algebraic  structures act on  Fock space, including   the Clifford algebra (or the algebra of fermions), the Heisenberg algebra $\mathcal {A}$,  the Virasoro algebra and the  infinite-dimensional Lie algebra  $gl_\infty$. Their actions  are closely related to  each other.

The algebra of fermions acts on the Fock space $\Fe$  by wedge  operators $\psi^+_k$  and contraction operators $\psi ^-_k$ ($k\in \bZ)$:
$$\psi^+_k \,( v_{i_1}\wedge v_{i_2}\wedge\cdots)= v_k\wedge v_{i_1}\wedge v_{i_2}\wedge\cdots,$$
and by
$$\psi^-_k\, ( v_{i_1}\wedge v_{i_2}\wedge\cdots)=
 \delta_{k, i_1 }v_{i_2}\wedge v_{i_3}\wedge\cdots
- \delta_{k, i_2 }v_{i_1}\wedge v_{i_3}\wedge\cdots
+\delta_{k, i_3 }v_{i_1}\wedge v_{i_2}\wedge\cdots-\dots\quad,
$$
The  operators  satisfy the  relations
\begin{align*}
\psi^+_k\psi^-_m+\psi^-_m\psi^+_k=\delta_{k,m},\quad
\psi^+_k\psi^+_m+\psi^+_m\psi^+_k=0,\quad
\psi^-_k\psi^-_m+\psi^-_m\psi^-_k=0.
\end{align*}
We combine the operators $\psi^\pm_k$  in generating functions  (formal distributions)
\begin{align}
\Psi^+(u)=\sum_{k\in \bZ} \psi^+_k u^{ k}\quad \text{ and } \quad  % \psi+=gamma+/u,    \psi -= \gamma-
\Psi^{-}(u)=\sum_{k\in \bZ} \psi^-_k u^{- k}. % CHECKED
\end{align}
Then the action of Heisenberg algebra  $\mathcal {A}$  on Fock  fermionic space  $\Fe$  can be introduced   with the help   of the normal ordered product of these formal distributions. Set
$$
\alpha(u)=:\Psi^+(u)\Psi^-(u):\, =\Psi^+(u)_+\Psi^-(u)-\Psi^-(u)\Psi^+(u)_-,
$$
where by definition of normal ordered product,
$$
\Psi^+(u)_+=\sum_{k\ge 1} \psi^+_ku^{k},\quad \Psi^+(u)_-=\sum_{k\le 0} \psi^+_ku^{k}. %CHECKED
$$
Then coefficients $\alpha_k$  of the formal distribution  $\alpha(u)=\sum \alpha_k u^{-k}$  and the central element $1$ satisfy  relations of Heisenberg algebra $\mathcal {A}$ (see e.g. \cite{Bom}, 16.3):
$$
[1,\alpha_k]=0,\quad   [\alpha_k,\alpha_m] =m\delta _{m,-k}  \quad (k,m\in \bZ).
$$

On the other hand,  there is also a natural action of  Heisenberg algebra $\mathcal {A}$  on Fock  boson space $\mathcal B^{(m)}=z^m\bC[ p_1, p_2,\dots ]$: % p.47
$$
\alpha_n=\frac{\partial}{\partial_{p_n}}, \quad \alpha_{-n}=n p_n,\quad  \alpha_0=m\quad(  n\in \bN,\, m\in \bZ).
$$

The boson\,--\,fermion correspondence    identifies  the spaces $\B^{(m)}$ and $ \Fe^{(m)}$ as  equivalent  $\mathcal{A}$-modules (see e.g. \cite{FBZ}, \cite{F}, \cite{FLM}, \cite{JM}, \cite{Bom}). The  construction of the correspondence relies   on the  interpretation of $\B=\bC[ p_1, p_2,\dots ]$ as a ring of  symmetric functions,  where  $p_k$'s  are  interpreted  as   $k$-th  (normalized) power sums. Then  each graded  component  $\B^{(m)}$ is viewed as a  ring of symmetric functions, which is  known to be the ring of polynomials in variables $p_k$'s.
The linear basis of elements  $v_\lambda=(v_{\lambda_1+m}\wedge v_{\lambda_2+m-1}\wedge v_{\lambda_3+m-2}\dots  )$ of $ \Fe^{(m)}$, labeled by partitions $\lambda=(\lambda_1,\ge \lambda_2,\ge \dots\ge \lambda_l\ge 0)$   corresponds to the linear basis  $z^m s_\lambda$ of $\B^{(m)}$  (see e.g. \cite{Bom}  Theorem 6.1), where $s_{\lambda}$ is the Schur functions associated with the partition $\lambda$.

The correspondence  carries the action of  operators  $\psi^\pm_k$  on  $\Fe$  to the action  on the graded space  $\oplus_m  \B^{(m)}$.  It can be  described by the (normalized)  generating functions $\Psi^\pm(u,m)$, written in so-called vertex operator form, which are  traditionally  written as
\begin{align}
     \Psi^+(u,m)= u^{m+1}z\exp \big(\sum_{j\ge 1}p_j u^j  \big) \exp \big(-\sum_{j\ge 1}\frac{ \partial_{p_j}}{j} {u^{-j}}\big),
\label{psi2}\\
     \Psi^-(u,m)= u^{-m}z^{-1} \exp \big(-\sum_{j\ge 1} {p_j} u^j  \big) \exp \big(\sum_{j\ge 1}\frac{\partial_{p_j}}{j}  {u^{-j}}\big).
\label{psi21}
\end{align}
The formulas (\ref{psi2}), (\ref{psi21})  can be simplified if one changes the   set of  generators of  the ring of symmetric functions. Namely,
introduce    generating functions $E(u)$, $H(u)$ for   the  operators of  multiplication   by elementary symmetric functions  $e_k$ and  complete  symmetric  functions $h_k$,
and   generating  functions $DE(u)$, $DH(u)$  for the corresponding adjoint operators
(see  Section \ref{A-sch} for definitions).
      Then  one can write
\begin{align}
    \Psi^+(u,m)&= u^{m+1}z \,H(u) DE\left(\frac{-1}{u}\right),\label{psi1}\\
       \Psi^-(u,m)&= u^{{-m}}z^{-1} \, E(-u) DH\left(\frac{1}{u}\right).\label{psi11}
\end{align}

\subsection{Symmetric functions}\label{A-sch}
For more details please refer e.g. to \cite{Md, Stan}.
The ring of symmetric functions in  variables $ (x_1,x_2,\dots)$ possesses  a linear basis of   Schur  functions $s_\lambda$,  labeled by partitions $\lambda=(\lambda_1\ge \lambda_2\ge \dots\ge \lambda_n\ge 0)$. By  definition,
$$
s_\lambda(x_1,x_2,\dots )=\frac {\det[x_i^{\lambda_j+n-j}]}{\det[x_i^{n-j}]}.
$$
Here the determinants can be taken over finitely many variables, and Schur functions satisfy the stability condition:
$s_{\lambda}(x_1, \ldots, x_n, x_{n+1})|_{x_{n+1}=0}=s_{\lambda}(x_1, \ldots, x_n)$,
which defines the symmetric function. For simplicity we will simply write the symmetric functions as 
polynomials in variables $x_1, x_2, \ldots$.

The  families of complete symmetric functions $h_r= s_{(r)}$ and elementary symmetric functions $e_r= s_{(1^r)}$  correspond to  one-row and  one-column
partitions:
\begin{align}
h_r(x_1,x_2\dots)=\sum_{1\le i_1\le \dots \le i_r<\infty} x_{i_1}\dots x_{i_r},\label{defh}\\
e_r(x_1,x_2\dots)=\sum_{1< i_1< \dots < i_r<\infty} x_{i_1}\dots x_{i_r}.\label{defe}
\end{align}

One can  write generating  functions  for  complete and elementary symmetric functions in the form
 \begin{align}
H(u)=\sum_{k\ge 0} h_k u^k= \prod_{i\ge 1} \frac{1}{1-x_i u},\quad
E(u)=\sum_{k\ge 0} e_k u^k=  \prod_{i\ge 1} (1+x_i u),\label{e:22}%\label{e:22}
 \end{align}
hence,  $H(u) E(-u)=1$.

 The ring of  symmetric functions $\Lambda$   is  a polynomial ring in algebraically  independent  variables $(h_1, h_2,\dots)$  or $(e_1, e_2,\dots )$:
 \begin{align}
 \Lambda=\bC[h_1, h_2,\dots]=\bC[e_1, e_2,\dots].
 \end{align}
 In particular, the  so-called Jacobi\,--\,Trudi  identity  establishes a relation  between  linear basis elements   (Schur  functions) and   generators :
\begin{align}
s_\lambda &= \det[h_{\lambda_i-i+j}]_{1\le i,j\le l}, \label{e:JT1} \\%\label{e:JT1}
s_\lambda &= \det[e_{\lambda^\prime_i-i+j}]_{1\le i,j\le l^\prime},\label{e:JT2}%\label{e:JT2}
\end{align}
where $\lambda^\prime$  is the conjugate of partition $\lambda$ with length $l^\prime$.

The
ring of symmetric functions  possesses a natural  scalar product, where the classical Schur functions $s_\lambda$ constitute
  an orthonormal basis: $\<s_\lambda,s_\mu\>=\delta_{\lambda, \mu}$. Then for any symmetric function $f$ one can
define the adjoint operator $D_f$ acting on the ring of symmetric functions by the standard rule:
$ \<D_fg, w\>=\<g,fw\>$, where $g,f,w\in \Lambda$.  The properties of  adjoint operators are described in  \cite{Md},  I.5.
We set
  \begin{align*}
  DE(u)= \sum_{k\ge 0} D_{e_k} u^k,\quad DH(u)= \sum_{k\ge 0} D_{h_k} u^k.
  \end{align*}

\subsection{Hall-Littlewood symmetric functions}\label{HL}

The main reference for these  polynomials is  \cite{Md}.
Let $\lambda$  be a  partition  of length  at most $n$.  Hall-Littlewood  polynomials  are defined  by
\begin{align*}
P_\lambda (x_1,\dots, x_n;t)=\left(\prod_{i\ge 0}\prod_{j=1}^{m(i)}\frac{1-t}{1-t^j}\right)\sum_{\sigma \in S_n}
\sigma\left(x_1^{\lambda_1}\dots x_n^{\lambda_n}
\prod_{\lambda_i>\lambda_j}\frac{x_i-tx_j}{x_i-x_j}\right),
\end{align*}
where $m(i)$  is  the number  of  parts of the partition $\lambda$  that are  equal to $i$, and
 $S_n$  is the symmetric group of $n$ letters.
 Specialization $t=0$  gives $P_\lambda(x_1,\dots , x_n; 0) =  s_\lambda(x_1,\dots, x_n)$,  and specialization  $t=-1$ provides the Schur $Q$-functions.

Hall-Littlewood  polynomials form a linear basis  of  the ring  $\Lambda[t]$ of polynomials in variable $t$ with coefficients in the ring $\Lambda$ of symmetric  functions.
Let $Q(u)$  be the generating function  for the  family  of Hall-Littlewood  functions  indexed by row partitions $\lambda= (r)$:
\begin{align}Q(u)=\sum_{r=0}^{\infty} P_{(r)}(x_1,\dots, x_n ;t) u^r.\label{QHL}
\end{align}
One can show that
\begin{align*}
Q(u)=\prod_{i}\frac{1-x_itu}{1-x_iu}= H(u) E(-tu),
\end{align*}
where $H(u)$  and $E(u)$  are the generating functions  in (\ref{e:22}).

The statement that   coefficients of the generating  function  (\ref{fHL}), where  $Q(u)$  is given by   (\ref{QHL}), are exactly  Hall-Littlewood  symmetric functions
is proved in
\cite{Md} III \,--\, (2.15).

\subsection{Schur  $Q$-functions}\label{App:q}
For the short  review of the of the  properties of   Schur  $Q$-functions we follow \cite{Ste1} and references  therein.
Specialization  $t=-1$ of Hall-Littlewood functions  defines Schur  $Q$-functions:
$Q_\lambda (x_1, x_2,\dots) = P_\lambda(x_1, x_2,\dots :-1)$.  The  historical  definition  \cite{Schur} of Schur  $Q$-functions without  reference to Hall-Littlewood  symmetric functions is as  follows.
For integers $n\in \bZ_{\ge 0}$  one defines  symmetric  functions $Q_n= Q_n (x_1, x_2,\dots) $  by
\begin{align}
Q(u)= \sum_{n=0}^{\infty} Q_n u^n=\prod_{i=1}^{\infty}\frac{1+x_iu}{1-x_iu}.\label{def:q}
\end{align}
For pairs of integers $m,n\in \bZ_{\ge 0}$  one defines  symmetric  functions $Q_{(m,n)}= Q_{(m,n)} (x_1, x_2,\dots) $  by
$$
Q_{(m,n)}= Q_mQ_n+2\sum_{s=1}^{n} (-1)^s Q_{m+s}Q_{m-s}.
$$
Then
 $Q_{(m,n)}= - Q_{(n,m)}$  and $Q_{(n,0)}= Q_n$.

Finally, consider a strict partition $\lambda=(\lambda_1>\dots >\lambda_l>0)$. If $l$  is odd, set $\lambda_{l+1}=0$, and replace  $l$  by $l+1$. Thus, we can assume that $l$  is always even.  Then   Schur $Q$-function indexed by $\lambda$  is defined by the Pfaffian of a skew-symmetric matrix
\begin{align*}
Q_\lambda=\mathrm{Pf}[ Q_ {(\lambda_i,\lambda_j)}]_{1\le i,j\le l}.
\end{align*}
In \cite{JP} the definition is  extended to  skew Schur  $Q$-functions  $Q_{\lambda/\mu}$. One can also define  skew Schur  $Q$-functions  in terms of   shifted  Young  tableau  \cite{Ste2}.
The subring of  symmetric functions,  generated  by  $(Q_1, Q_2,\dots)$  is known  to be a polynomial ring  $\bC[p_1,p_3, p_5,\dots]$, where
$p_k =\sum_i x_i^k$
is the $k$-th power sum.  Schur  $Q$-functions  $Q_\lambda$, labeled by  strict  partitions,  form a linear basis of this  subring.

\subsection{Shifted Schur functions.}\label{ashift}
We follow   notations  and  definitions of \cite{OO1}.
The algebra of shifted symmetric functions $\B^*$  is a deformation  of  the classical algebra  of symmetric  functions, which has many   applications  in representation theory.
In particular,  shifted symmetric  functions are well-known in the study of the centers  of universal enveloping algebras,  of Capelli-type identities,  of asymptotic  characters  for unitary groups and symmetric groups, and  for the connections  to  representation theory of infinite-dimensional  quantum groups and Yangians.  Namely,  the  Harish-Chandra isomorphism  identifies  the center of  the universal enveloping  algebra  of the general linear Lie algebra $gl_n(\bC)$  with the algebra of  shifted symmetric functions,  sending  a central element to its eigenvalue on a highest weight module.
Then one can consider  a distinguished  basis of the center, and  the images of the elements of that basis  under the isomorphism are  exactly   the shifted Schur functions.

Combinatorially  a shifted Schur polynomial   $s^*_\lambda(x_1,\dots, x_n)$ can  be defined as
 a ratio of determinants
\begin{align*}
s_\lambda^*(x_1,\dots, x_n)=\frac{\det(x_i+n-i|\lambda_j+n-j)}{\det(x_i+n-i|n-j)},
\end{align*}
where shifted  powers $(x|k)$ are  defined by (\ref{falf}).
 The stability property of shifted Schur  polynomials allows one  to introduce the shifted Schur functions, which we denote as
$s_\lambda^*=s_\lambda^*(x_1, x_2,\dots)$.
In  particular, the  complete  shifted Schur functions $h^*_r=s^*_{(r)}$ are
\begin{align}
h^*_r(x_1,x_2,\dots )=\sum_{1\le i_1 \le \dots \le i_r <\infty} (x_{i_1}-r+1)(x_{i_2}-r+2)\cdots x_{i_r},\label{h^*}
\end{align}
and the  elementary shifted Schur functions $e^*_r=s^*_{(1^r)}$ are
\begin{align}
e^*_r(x_1,x_2,\dots )=\sum_{1\le i_1 <\dots < i_r <\infty} (x_{i_1}+r-1)(x_{i_2}+r-2)\cdots x_{i_r}.\label{e^*}
\end{align}
By Corollary 1.6. in \cite{OO1},  shifted Schur  functions $s^*_\lambda$  form a linear basis in the ring of shifted symmetric functions, which is also a polynomial  ring in shifted  complete or elementary symmetric  functions:
$$
\B^*=\bC[h^*_1, h^*_2,\dots]=\bC[e^*_1, e^*_2,\dots].
$$
 Theorem 13.1  in  \cite{OO1} states that
\begin{align}
s^*_\lambda= \det [\tau^{j-1}h^*_{\lambda_i-i+j}]_{1\le i,j\le l}, \quad
s^*_\lambda= \det [\tau^{1-j}e^*_{\lambda^\prime_i-i+j}]_{1\le i,j\le m},\label{shiftJT2}
\end{align}
for arbitrary $l,m$  such that $l\ge l(\lambda)$, $m\ge \lambda_1$,
where $\tau$ is the automorphism of  the algebra  of  shifted Schur  functions, defined by the formula
$$ \tau(h^*_k)=h^*_k+(k-1)h^*_{k-1},\quad \tau^{-1}(e^*_k)=e^*_k+(k-1)e^*_{k-1}. $$

\end{document}